\documentclass[12pt]{amsart}
\usepackage{amsfonts}
\usepackage[margin=1in, centering]{geometry}
\usepackage{amssymb,latexsym,amsmath,extarrows, mathrsfs, amsthm}
\usepackage{graphicx}
\usepackage[colorlinks, linkcolor=blue]{hyperref}
\usepackage{color}
\usepackage{wasysym}
\usepackage{float}
\usepackage{stucky}

\newtheorem{theorem}{Theorem}[section]
\newtheorem{lemma}[theorem]{Lemma}
\newtheorem{proposition}[theorem]{Proposition}

\theoremstyle{remark}

\allowdisplaybreaks
\email{m. karras@univ-dbkm.dz}
\email{ling.li.China@hotmail.com}
\email{Joshua.Stucky@uga.edu}
\keywords{ integer part, exponent pairs, three-dimensional exponential sums}
\subjclass[2020]{11A25, 11L07, 11N37}

\begin{document}
\title{Hyperbolic Summation for Fractional Sums}
\author{Meselem KARRAS, Ling LI, and Joshua STUCKY }
\address{ Meselem Karras, Department of Mathematics, University of Djilali
	Bounaama, FIMA Laboratory, Khemis Miliana, Algeria.}
\address{Li Ling, School of Mathematics and Statistics, Qingdao University,
	308 Ningxia Road, Shinan District, Qingdao, Shandong, China.}
\address{Joshua STUCKY, University of Georgia, UGA, Department of Mathematics, United States}
\begin{abstract}
Let $f(n)$ be an arithmetic function with $f(n) \ll n^\alpha$ for some $\alpha\in[0,1)$ and let $\lfloor .\rfloor $ denote the integer part function. In this paper, we evaluate asymptotically the sums
\begin{equation*}
\sum_{n_{1}n_{2}\leq x}f \left( \left\lfloor \frac{x}{n_{1}n_{2}}%
\right\rfloor \right),
\end{equation*}
we use the estimation of three-dimensional exponential sums due to Robert and Sargos.
\end{abstract}

\maketitle

\section{Introduction}

For an integer $n\geq 1$, let $\tau(n)$ be the number of positive divisors of $n$. A classical result of Dirichlet states that
\begin{align}  \label{eq: asymptotic formula for the sum of tau(n)}
\sum_{n\leq x}\tau(n)=x(\log{x}+2\gamma-1)+\Delta(x)
\end{align}
where is $\gamma$ is the Euler's constant and $\Delta(x)$ is an error term. Dirichlet proved that $\Delta(x) \ll x^{1/2}$. Improved estimated for $\Delta(x)$ have been given by numerous authors (\cite{Voronoi,Van der Corput1,Van der Corput2,Kolesnik,Iwaniec}), with the current record being Huxley's \cite{Huxley2} estimate: for any $\ep > 0$,
\begin{equation} \label{eq: Bourgain and Watt's result}
\Delta(x)\ll x^{\frac{131}{416}+\ep}.
\end{equation}

By the definition of $\tau(n)$, we note that 
\begin{equation*}
\sum_{n\leq x}\tau(n)=\sum\limits_{kl\leq x}1=\sum_{k\leq x}\sum_{l\leq 
\frac{x}{k}}1=\sum_{n\leq x}\left\lfloor\frac{x}{n}\right\rfloor.
\end{equation*}
Thus, we can consider \eqref{eq: asymptotic formula for the sum of tau(n)}
as an asymptotic formula for the fractional sum $\sum_{n\leq x}\left \lfloor
x/n\right\rfloor$, where $\lfloor t\rfloor$ is the largest integer not
exceeding $t$. With this viewpoint, Bordell\`{e}s, Dai, Heyman, Pan and
Shparlinski \cite{BDHPS} investigated the more general sums
\[
S_f(x):=\sum_{n\leq x}f\left(\left\lfloor\frac{x}{n}\right\rfloor\right)
\]
for arithmetic functions $f$ satisfying various growth conditions. Later, Wu \cite{W} and Zhai \cite{Z} independently showed that if 
\begin{equation*}
f(n)\ll n^{\alpha}\left(\log{n}\right)^{\theta}
\end{equation*}
for some $\alpha\in[0,1)$ and $\theta\geq 0$, then

\begin{align}  \label{eq: S_f(x) asymptotic formula of Zhai and Wu}
S_{f}(x)=x\sum_{n=1}^{\infty}\frac{f(n)}{n(n+1)}+O\left(x^{(1+\alpha)/2}%
\left(\log{x}\right)^{\theta}\right).
\end{align}

It is possible to improve \eqref{eq: S_f(x) asymptotic formula of Zhai and Wu} for functions $f$ satisfying certain decomposition identities. For example, if $f=\Lambda$, the von Mangoldt function, Ma and Wu \cite{MW} proved that 
\begin{align*}
\sum_{n\leq x}\Lambda\left(\left\lfloor\frac{x}{n}\right\rfloor\right)=x%
\sum_{n=1}^{\infty}\frac{\Lambda(n)}{n(n+1)}+O\left(x^{\frac{35}{71}%
+\varepsilon}\right)
\end{align*}
for any $\varepsilon>0$. Here the ``decomposition identity'' needed is Vaughan's identity. The exponent $35/71$ was then improved to $9/19$ by Liu, Wu and Yang \cite{LWZ}. If $f=\tau$, Ma and Sun \cite{MS} showed that
\[
S_{\tau}(x)=\sum_{n\leq x}\tau\left(\left\lfloor\frac{x}{n}%
\right\rfloor\right)=x\sum_{n=1}^{\infty}\frac{\tau(n)}{n(n+1)}+O\left(x^{\frac{11}{23}+\varepsilon}\right).
\]
The exponent $11/23$ was then improved to $19/40$ and $9/19$, respectively,
by Bordell\`{e}s \cite{Bordell} and Liu, Wu and Yang \cite{LWZ2}, and finally to $5/11$ by the third author \cite{Stucky}. 

Motivated by recent results, it is important to study sums of the form 
\begin{equation*}
T_{f,r}(x)=\sum_{n_1n_2\cdots n_r\leq x}f\left(\left\lfloor\frac{x}{n_1n_2\cdots
n_r}\right\rfloor\right) = \sum_{n\leq x} f\pth{\floor{\frac{x}{n}}}\tau_r(n),
\end{equation*}
where $r$ is a fixed positive integer and $\tau_r$ is the number of ways of writing $n$ as a product of $r$ positive integers. In this paper, we consider the case $r=2$ and put
\begin{equation*}
T_f(x) = T_{f,2}(x) = \sum_{n\leq x} f\pth{\floor{\frac{x}{n}}}\tau(n)
\end{equation*}
Supposing that $\Delta(x) \ll x^{\theta+\ep}$ and that $f(n) \ll n^\alpha$ for some $\alpha\in[0,1)$, it is not hard to show (see Section \ref{sec: Proof of (1.4)}) that
\begin{equation}\label{eq:Asymptotic}
T_f(x)=C_1(f)x\log {x}+C_2(f) x+O\pth{x^{\frac{\alpha(1-\theta)+1}{2-\theta}+\ep}},
\end{equation}
where 
\[
C_1(f)=\sum_{d\geq 1}\frac{f(d)}{d(d+1)}, \qquad C_2(f) = C_1(2\gamma-1) - C_3(f),
\]
and 
\[
C_3(f)=\sum_{d\geq 1} f(d)\left( \frac{\log {d}}{d}-\frac{\log (d+1)}{d+1}
\right).
\]
Note that both sums converge since $\alpha < 1$. Huxley's bound \eqref{eq: Bourgain and Watt's result} then gives
\[
T_f(x)=C_1(f)x\log {x}+C_2(f) x+O\pth{x^{\frac{416 + 285\alpha}{701} + \ep}}.
\]
It is known that $\Delta(x) \gg x^{1/4}$ for infinitely many values of $x$, and it is conjectured that $\Delta(x) \ll x^{1/4+\ep}$ for all $x$. Taking $\theta = \frac{1}{4}$ in \eqref{eq:Asymptotic}, one conjectures that
\[
T_f(x)=C_1(f)x\log {x}+C_2(f) x+O\pth{x^{\frac{4 + 3\alpha}{7} + \ep}}.
\]
Using an estimate for three-dimensional exponential sums due to Robert and Sargos \cite{Robert}, we are able to prove this result unconditionally.

\begin{theorem}\label{theorem:1}
Let $f(n)$ be an arithmetic function with $f(n) \ll n^\alpha$ for some $\alpha\in[0,1)$, and let $C_1(f),C_2(f)$ be as above. For any $x\geq 2$ and $\varepsilon >0$, we have 
\begin{equation}\label{eq: asymptotic for T(x)}
T_f(x)=C_1(f)x\log {x}+C_2(f) x+O\pth{x^{\frac{4 + 3\alpha}{7} + \ep}}.
\end{equation}
\end{theorem}

\noindent\textbf{Notation:} The Landau-Vinogradov symbols $\ll, \gg, O$ have their usual meanings. We write $\lfloor x\rfloor$ to denote largest integer not exceeding $x$, $e(x)$ to denote $e^{2\pi ix}$, $\psi(t) = t - \floor{t} - \half$. As well, we write $d\sim D$ to denote $D<d \leq 2D$. The sum $\sumd_{D\leq x}$ means a sum over $D = 2^k$ with $2^k \leq x$.

\section{Preliminaries}

In this section, we state some lemmas that will be needed in the proof of Theorem \ref{theorem:1}. For the error term $\Delta(x)$ in  \eqref{eq: asymptotic formula for the sum of tau(n)}, we have the following expression.

\begin{lemma}
\cite[Theorem 4.5]{GK}\label{lemma:1} For any $x\in \mathbb{R}$, we have 
\begin{equation*}
\Delta(x)=-2\sum_{n\leq \sqrt{x}}\psi\left(\frac{x}{n}\right)+O(1).
\end{equation*}
\end{lemma}

The function $\psi(x)$ is periodic with period $1$ and so can be expanded into a Fourier series. We need the following truncated version due to Vaaler.

\begin{lemma}
\cite[Lemma 4.1]{BDHPS}\label{lemma:2} For $x$ be real and $H\geq 1$, we
have 
\begin{equation*}
\psi(x)=-\sum_{1\leq |h|\leq H}\Phi\left(\frac{h}{H+1}\right)\frac{e(hx)}{%
2\pi ih}+R_{H}(x),
\end{equation*}
where $\Phi(t):=\pi t(1-|t|)\cos(\pi t)+|t|$, and the error term $R_{H}(x)$
satisfies 
\begin{equation*}
\left|R_{H}(x)\right|\leq \frac{1}{2H+2}\sum_{|h|\leq H}\left(1-\frac{|h|}{%
H+1}\right)e(hx).
\end{equation*}
\end{lemma}

Lastly, we need the following lemma on certain three-dimensional exponential sums.

\begin{lemma}
\cite[Therom 3]{Robert}\label{lemma:3} Let 
\begin{equation*}
S=\sum_{h=H+1}^{2H}\sum_{n=N+1}^{2N}\left|\sum_{M<m\leq 2M}e\left(X\frac{%
m^{\alpha}h^{\beta}n^{\gamma}}{M^{\alpha}H^{\beta}N^{\gamma}}%
\right)\right|^{*}
\end{equation*}
where $H$, $N$ and $M$ are positive integers, $X \geq 1$ is a real number, $\alpha$, $\beta$ and $\gamma$ are fixed real number such that $\alpha(\alpha-1)\beta\gamma\neq0$, and 
\begin{equation*}
\left|\sum_{1\leq n\leq N}z_n\right|^{*}=\max_{1\leq N_1\leq N_2\leq
N}\left|\sum_{n=N_1}^{N_2}z_n\right|.
\end{equation*}
Then we have 
\begin{equation*}
S\ll \left(HNM\right)^{1+\varepsilon}\left\{\left(\frac{X}{HNM^2}%
\right)^{1/4}+\frac{1}{M^{1/2}}+\frac{1}{X}\right\}.
\end{equation*}
\end{lemma}

As a consequence of the above three lemmas, we deduce the following estimate for the average of $\Delta$ over a sequence of monomials.

\begin{proposition}\label{prop:DeltaBound}
For $X, D \geq 1$ and $\alpha\neq 0$, we have
\[
\sum_{d\sim D} \abs{\Delta\pth{X \frac{d^\alpha}{D^\alpha}}} \ll \pth{X^{3/8}D^{3/4}  +X^{1/4}D }(XD)^\ep,
\]
where the implied constant depends only on $\alpha$ and $\ep$.
\end{proposition}

We note that this agrees with the estimate one obtains from the conjectured bound $\Delta(x) \ll x^{1/4+\ep}$ so long as $D \geq \sqrt{X}$.

\begin{proof}
Let $H = 2^k$ be such that $2^k \leq \sqrt{X} < 2^{k+1}$. By Lemmas \ref{lemma:1} and \ref{lemma:2}, we have that for certain coefficients $c_h$ with $\abs{c_h} \ll \frac{1}{h}$, 
\[
\begin{aligned}
\abs{\Delta\pth{X \frac{d^\alpha}{D^\alpha}}} &\ll 1 + \sumabs{\sum_{l\leq \sqrt{X \frac{d^\alpha}{D^\alpha}}} \sumpth{\frac{1}{H} + \sum_{h\leq H} c_h e\pth{X\frac{hd^\alpha}{l D^\alpha}}}} \\
&\ll 1 + \sum_{h\leq H} \frac{1}{h} \sumabs{\sum_{l\leq \sqrt{X \frac{d^\alpha}{D^\alpha}}} e\pth{X\frac{hd^\alpha}{l D^\alpha}}}.
\end{aligned}
\]
We divide the sums over $l$ and $h$ into dyadic sums with $l\sim L$ and $h \sim H'$, where $L$ and $H'$ are powers of 2. Then the sum on the last line is
\[
\begin{aligned}
&\ll \sumd_{H' \leq H} \frac{1}{H'} \sumd_{L\leq \sqrt{2^\alpha X}} \sum_{h\sim H'} \sumabs{\sum_{\substack{l\sim L \\ l \leq \sqrt{X \frac{d^\alpha}{D^\alpha}}}} e\pth{X\frac{hd^\alpha}{l D^\alpha}}}\\
&\ll \sumd_{H' \leq H} \frac{1}{H'} \sumd_{L\leq \sqrt{2^\alpha X}} \sum_{h\sim H'} \sumabs{\sum_{\l\sim L} e\pth{X\frac{hd^\alpha}{l D^\alpha}}}^*
\end{aligned}
\]
by the definition of $\abs{\cdot}^*$. Thus
\begin{equation}\label{eq:DeltaBoundDyadic}
\sum_{d\sim D} \abs{\Delta\pth{X \frac{d^\alpha}{D^\alpha}}} \ll D+ \sumd_{H' \leq H} \sumd_{L\leq \sqrt{2^\alpha X}}  \sumpth{\frac{1}{H'}\sum_{h\sim H'} \sum_{d\sim D} \sumabs{\sum_{\l\sim L} e\pth{X\frac{hd^\alpha}{l D^\alpha}}}^*}
\end{equation}
By Lemma \ref{lemma:3}, we have
\[
\begin{aligned}
\frac{1}{H'}\sum_{h\sim H'} \sum_{d\sim D} \sumabs{\sum_{\l\sim L} e\pth{X\frac{hd^\alpha}{l D^\alpha}}}^* &\ll DL \pth{\fracp{X}{D L^3}^{1/4} + \frac{1}{L^{1/2}} + \frac{L}{XH'}}(XD)^\ep \\
&\ll \pth{X^{1/4} D^{3/4} L^{1/4} + D L^{1/2} + \frac{L^2D}{X}} (XD)^\ep.
\end{aligned}
\]
Combining this with \eqref{eq:DeltaBoundDyadic}, we find that
\[
\begin{aligned}
\sum_{d\sim D} \abs{\Delta\pth{X \frac{d^\alpha}{D^\alpha}}} &\ll D + \sumd_{H' \leq H} \sumd_{L\leq \sqrt{2^\alpha X}} \pth{X^{1/4} D^{3/4} L^{1/4} + D L^{1/2} + \frac{L^2D}{X}} (XD)^\ep  \\
&\ll \pth{X^{3/8} D^{3/4}  +X^{1/4}D }(XD)^\ep.
\end{aligned}
\]
\end{proof}

From this, we deduce the following special case needed for the proof of Theorem \ref{theorem:1}.

\begin{proposition}\label{prop:NeededDelta}
For $x \geq 1$, $D \leq x$, and $\delta \in \set{0,1}$, we have
\[
\sum_{d\sim D} \abs{\Delta\fracp{x}{d+\delta}} \ll \pth{x^{3/8} D^{3/8} + x^{1/4} D^{3/4} }x^\ep,
\]
where the implied constant depends only on $\ep$ and $\delta$.
\end{proposition}

\begin{proof}
If $\delta = 0$, we apply Proposition \ref{prop:DeltaBound} with $\alpha = -1$ and $X = \frac{x}{D}$. If $\delta = 1$, then by positivity, we have
\[
\sum_{d\sim D} \abs{\Delta\fracp{x}{d+\delta}}  = \sum_{D - 1 < d \leq 2D-1} \abs{\Delta\fracp{x}{d}} \leq \sumpth{\sum_{D/2 < d \leq D} + \sum_{D < d \leq 2D}} \abs{\Delta\fracp{x}{d}},
\]
and the required bound follows by applying Proposition \ref{prop:DeltaBound} to each of the two sums over $d$ with $X = \frac{x}{D}$ and $X = \frac{2x}{D}$.

\end{proof}

\section{Proof of Theorem \ref{theorem:1}}\label{sec: Proof of (1.4)}

Let $N\in [1,x)$ be a parameter that will be chosen later. We split $T_f(x)$ into two parts: 
\begin{equation}\label{equation:S}
T_f(x)=T_1(x)+T_2(x),
\end{equation}
where
\[
T_1(x):=\sum_{n\leq N}f\left(\left\lfloor\frac{x}{n}\right\rfloor\right)%
\tau(n),\qquad T_2(x):=\sum_{N<n\leq x}f\left(\left\lfloor\frac{x}{n}%
\right\rfloor\right)\tau(n).
\]
Since $\tau(n)\ll n^{\varepsilon}$ for any $\ep > 0$, we have

\begin{equation}\label{equation:S1}
T_1(x) \ll \sum_{n\leq N} \frac{x^{\alpha+\ep}}{n^\alpha} \ll x^{\alpha+\ep} N^{1-\alpha}.
\end{equation}
If $d=\lfloor x/n\rfloor$, then $x/n-1< d \leq x/n$ and $x/(d+1)<n \leq x/d$, so
\[
T_2(x) =\sum_{N<n\leq x}\sum_{\left\lfloor\frac{x}{n}\right\rfloor=d}f(d)%
\tau(n) =\sum_{d\leq \frac{x}{N}}f(d)\sum_{\frac{x}{d+1}<n\leq \frac{x}{d}%
}\tau(n).
\]
By \eqref{eq: asymptotic formula for the sum of tau(n)}, we
obtain
\begin{equation}
\label{equation:S2}
T_2(x)=\sum_{d\leq \frac{x}{N}}f(d)\left(\sum_{n\leq \frac{x}{d}%
}\tau(n)-\sum_{n\leq \frac{x}{d+1}}\tau(n)\right)=T_{21}(x)-T_{22}(x)+T_{\Delta}(x),
\end{equation}
where
\[
\begin{aligned}
T_{21}(x)&:=x(\log{x}+2\gamma-1)\sum_{d\leq \frac{x}{N}}\frac{f(d)}{d(d+1)}, \\
T_{22}(x)&:=x\sum_{d\leq \frac{x}{N}}f(d)\left(\frac{\log{d}}{d}-\frac{\log{(d+1)}}{d+1}\right), \\
T_{\Delta}(x)&:=\sum_{d\leq \frac{x}{N}}f(d)\left(\Delta\left(\frac{x}{d}\right)-\Delta\left(\frac{x}{d+1}\right)\right).
\end{aligned}
\]
Since $f(d) \ll d^\alpha$ with $\alpha < 1$, we have
\[
\sum_{d> \frac{x}{N}}\frac{f(d)}{d(d+1)} \ll \sum_{d> \frac{x}{N}} d^{\alpha-2} \ll \fracp{x}{N}^{\alpha-1},
\]
and similarly
\[
\sum_{d> \frac{x}{N}} f(d)\left(\frac{\log{d}}{d}-\frac{\log{(d+1)}}{d+1}\right) \ll \fracp{x}{N}^{\alpha-1}\log x.
\]
Therefore
\[
T_{21}(x) - T_{22}(x) = C_1(f)x\log x + C_2(f)x + O\pth{x^{\alpha+\ep}N^{1-\alpha}},
\]
and so
\begin{equation}\label{eq:T expansion}
T_f(x) = C_1(f)x\log x + C_2(f)x + O\pth{x^{\alpha+\ep}N^{1-\alpha}} + T_\Delta(x).
\end{equation}

To prove \eqref{eq:Asymptotic}, we note that the estimate $\Delta(x) \ll x^{\theta+\ep}$ immediately gives
\[
T_\Delta(x) \ll \sum_{d\leq \frac{x}{N}} d^{\alpha} \fracp{x}{d}^{\theta+\ep} \ll x^{1+\alpha+\ep} N^{\theta-1-\alpha},
\]
and \eqref{eq:Asymptotic} follows by choosing $N = x^{\frac{1}{2-\theta}}$.

Finally, to prove Theorem \ref{theorem:1}, we note that
\[
\abs{T_\Delta(x)} \ll \sumd_{D\leq \frac{x}{N}} D^\alpha \sum_{d\sim D} \sumpth{\abs{\Delta\fracp{x}{d}} + \abs{\Delta\fracp{x}{d+1}}},
\]
so that Proposition \ref{prop:NeededDelta} gives
\[
\begin{aligned}
\abs{T_\Delta(x)} &\ll \sumd_{D\leq \frac{x}{N}} D^\alpha\pth{x^{3/8}D^{3/8} + x^{1/4} D^{3/4}}x^\ep \\
&\ll x^{3/4+\alpha+\ep} N^{-3/8-\alpha} + x^{1+\alpha+\ep} N^{-3/4-\alpha}.
\end{aligned}
\]
The second term dominates so long as $N \leq x^{2/3}$. In this case, combining the above estimate with \eqref{eq:T expansion}, we have
\[
T_f(x) = C_1(f)x\log x + C_2(f)x + O\pth{x^{\alpha+\ep}N^{1-\alpha} + x^{1+\alpha+\ep} N^{-3/4-\alpha}}.
\]
 Choosing $N = x^{4/7}$, we complete the proof of Theorem \ref{theorem:1}.\\

\section{Acknowledgements}
The authors would like to thank Professor Kui Liu for his suggestions to improve this paper.\\

\end{document}